\documentclass[a4paper,oneside,leqno, 11pt]{article}
\usepackage{hyperref}
\hypersetup{
  colorlinks   = true, 
  urlcolor     = blue, 
  linkcolor    = blue, 
  citecolor   = purple 
}
\usepackage[utf8]{inputenc}
\usepackage{fancyhdr}
\usepackage{amsmath}
\usepackage{amsthm}
\usepackage{amsfonts, amssymb}
\usepackage{amsthm}
\usepackage{mathrsfs}
\usepackage{mathtools}
\usepackage{graphics}
\usepackage{amssymb}
\usepackage{relsize}
\usepackage[all]{xy}
\usepackage{color}
\usepackage{tikz-cd}
\usepackage{tikz}
\usetikzlibrary{calc}
\usepackage{listings}
\usepackage{geometry}

\usepackage{subcaption}
\usepackage{float}

\newcommand\NN{\mathbb{N}}
\newcommand\ZZ{\mathbb{Z}}
\newcommand\QQ{\mathbb{Q}}

\newcommand\CC{\mathbb{C}}
\newcommand\PP{\mathbb{P}}

\newcommand\C{\mathcal{C}}

\newcommand\OO{\mathcal{O}}
\newcommand\X{\mathfrak{X}}
\newcommand\HH{\mathcal{H}}

\newcommand\GG{\mathcal{G}}
\newcommand\Q{\mathcal{Q}}
\newcommand\g{\mathfrak{g}}
\newcommand\ov{\overline}
\newcommand\wt{\widetilde}
\newcommand\FF{\mathcal{F}}

\newcommand\KK{\mathcal{K}}
\newcommand\LL{\mathcal{L}}

\newcommand\TT{\mathcal{T}}
\newcommand\UU{\mathcal{U}}
\newcommand\VV{\mathcal{V}}
\newcommand\WW{\mathcal{W}}
\newcommand\XX{\mathfrak{X}}
\newcommand\Sl{\mathfrak{sl}}
\newcommand\aff{\mathfrak{aff}}
\newcommand\aut{\mathfrak{aut}}
\newcommand\so{\mathfrak{so}}

\newcommand\INV{\mathfrak{Inv}}

\newcommand\SUB{\mathfrak{Sub}}
\newcommand\inv{\text{Inv}}

\newcommand\Quot{\text{Quot}}
\newcommand\Aut{\text{Aut}}
\newcommand\Hom{\text{Hom}}

\DeclareMathOperator{\codim} {codim}

\DeclareMathOperator{\im} {Im}

\newcounter{Theorem}

\newtheorem{teo}[Theorem]{Theorem} 
\newtheorem{prop}[Theorem]{Proposition}
\newtheorem{lema}[Theorem]{Lemma}
\newtheorem{corollary}[Theorem]{Corollary}

\newtheorem{nteo}{Theorem}

\newtheorem{ncorollary}[nteo]{Corollary}

\theoremstyle{definition}
\newtheorem*{nobs}{Remark}
\newtheorem*{ndeff}{Definition}

\theoremstyle{definition}
\newtheorem{deff}[Theorem]{Definition}
\newtheorem{obs}[Theorem]{Remark}
\newtheorem{ej}[Theorem]{Example}

\newenvironment{dem}
	{\noindent\textit{Proof.}}{\qed \vspace{.5cm}}

\definecolor{orange}{RGB}{255,127,0}

\title{Moduli spaces of Lie algebras and foliations}
\author{Sebastián Velazquez\footnote{King's College, London, United Kingdom. The author was also supported by CONICET-Argentina, CNPq-Brazil and EPSRC-United Kingdom.\\ E-mail adress: sebastian.velazquez@kcl.ac.uk}}
\date{}

\begin{document}
\maketitle

\begin{abstract}
Let $X$ be a smooth projective variety over the complex numbers and $S(d)$ the scheme parametrizing $d$-dimensional Lie subalgebras of $H^0(X,\TT X)$. This article is dedicated to the study of the geometry of the moduli space $\inv$ of involutive distributions on $X$ around the points $\FF\in \inv$ which are induced by Lie group actions. For every $\g \in S(d)$ one can consider the corresponding element $\FF(\g)\in \inv$, whose generic leaf coincides with an orbit of the action of $\exp(\g)$ on $X$. We show that under mild hypotheses, after taking a stratification $\coprod_i S(d)_i\to S(d)$ this assignment yields an isomorphism $\phi:\coprod_i S(d)_i\to \inv$ locally around $\g$ and $\FF(\g)$. This gives a common explanation for many results appearing independently in the literature. We also construct new stable families of foliations which are induced by Lie group actions.
\end{abstract}

\section{Introduction}
Let $X$ be a smooth projective variety of dimension $n$. A singular holomorphic foliation $\FF$ on $X$ is said to be an \emph{$L$-foliation} if there exists a complex Lie group $G$ acting on $X$ with orbits of constant dimension in codimension 1 whose generic orbit coincides with the generic leaf of $\FF$. Observe that such a foliation is uniquely determined by the finite-dimensional Lie algebra $H^0(X,\TT\FF)$ of global vector fields tangent to it. These foliations have played a central role in the study of singular foliations and their moduli spaces. For instance, in the literature there are many results regarding the  \emph{stability} of different families of $L$-foliations, i.e., showing irreducible components of moduli spaces of foliations whose generic point is an $L$-foliation, see for example \cite{notej}, \cite{CLN}, \cite{CP} and \cite{fano3folds} among others. In this work we present a common framework for these phenomena and generalize many of them to broader classes of varieties and group actions. We will do so by means of comparing the moduli problem associated to the  subalgebra $H^0(X,\TT\FF)\subseteq H^0(X,\TT X)$ and the one corresponding to $\TT\FF\subseteq \TT X$ as an involutive coherent subsheaf. 

We will now give an overview of our approach. Let us use the notation $L=H^0(X,\TT X)$ for the algebra of global vector fields on $X$.
First, observe that every $L$-foliation $\FF$ is induced by the action of the connected Lie group $\exp(\g)\subseteq \Aut(X)$ for $\g=H^0(X,\TT\FF)$. Conversely, for an arbitrary Lie subalgebra $\g\subseteq H^0(X,\TT X)$ 
we will denote by $\FF(\g)$ the exact sequence
$$\hspace{0.1cm}0\rightarrow \TT\FF(\g) \rightarrow \TT X\rightarrow N_{\FF(\g)} \rightarrow 0, $$
where $\TT \FF(\g)$ is the (involutive) subsheaf of $\TT X$ generated by $\g$. Observe that if $\g$ acts infinitesimally on $X$ with orbits of constant dimension in codimension $1$, this is just the foliation induced by $\g$.
We will say that a subalgebra $\g\subseteq L$ is \emph{maximal} if the inclusion $\g\subseteq H^0(X,\TT\FF(\g))$ is an equality. In order to keep the notation simple, for every scheme $S$ we will denote by $\TT X_S$ the relative tangent sheaf $\TT_S(X\times S)$.

\begin{ndeff}A \emph{family of codimension $q$ foliations} on $X$ over a scheme $S$ is an exact sequence 
$$\FF: \hspace{0.1cm}0\rightarrow \TT\FF \rightarrow \TT X_S \rightarrow N_\FF \rightarrow 0 $$
such that $\TT\FF$ is closed under the Lie bracket of relative vector fields and its normal sheaf $N_\FF$ is torsion-free on the fibers and of rank $q$. The family is said to be \emph{flat} if $N_\FF$ is flat over $S$.
\end{ndeff}
Let $\inv\subseteq \Quot(\TT X)$ be the locally closed subscheme representing the functor assigning to each scheme $S$ the set of flat families of foliations on $X$ over $S$
(for more details see Definition \ref{definv}). This space was introduced in \cite{Pou} and \cite{Quall} independently. The problem of classifying foliations and their flat families corresponds to understanding the geometry of this space, which is our main object of study. Let us also denote $\overline{\inv}\subseteq \Quot(\TT X)$ the subscheme representing the functor of flat families of coherent involutive subsheaves of $\TT X$ (this is, families whose normal sheaf may not be torsion-free).

\begin{ndeff}A \emph{family of $d$-dimensional Lie subalgebras of vector fields with base $S$} is a locally free subsheaf $\GG\subseteq L\otimes_\CC \OO_S$ such that $[\GG,\GG]\subseteq \GG$, where $[\hspace{0.05cm},]$ stands for the $\OO_S$-linear extension of the Lie bracket on $L$. 
\end{ndeff}
Proposition \ref{lierep} shows that the corresponding \emph{functor of $d$-dimensional Lie subalgebras of $L$} is represented by a subvariety $S(d)$ of the Grassmannian $Gr(L,d)$. A study of some local aspects of its geometry can be found in \cite{Rich} and the references within. Of course, $S(d)$ is equipped with a universal family $\GG_d$ of $d$-dimensional subalgebras of $L$. Now let
$$\FF_d: \hspace{0.1cm}0\rightarrow \TT\FF_d \rightarrow \TT X_S \rightarrow N_{\FF_d} \rightarrow 0 $$
be the family of coherent involutive subsheaves over $S(d)$ such that the restriction $(\FF_d)_\g$ to the fiber $X_\g$ over a point $\g\in S(d)$ satisfies $(\FF_d)_\g=\FF(\g)$. Since this family may not be flat, we shall consider the flattening stratification 
$$\coprod_i S(d)_i \to S(d)$$
associated to the sheaf $N_{\FF_d}$. We thus get over each stratum $S(d)_i$ a flat family $\FF_i:=\FF\vert_{S(d)_i}$ and therefore a canonical morphism $\phi_i: S(d)_i\to \ov{\inv}$. It is worth mentioning that these schemes carry natural actions of $\Aut(X)$ and $\phi_i$ turns out to be $\Aut(X)$-equivariant. An element is said to be \emph{rigid} if its orbit contains a non-empty Zariski open. We will denote by 
$$\phi:\coprod_i S(d)_i\to \overline{\inv}$$
the morphism induced by the  $\phi_i$'s. In these terms, the stability of $L$-foliations corresponds to the local surjectivity of $\phi$. Our main result states that under mild hypotheses this is actually a local isomorphism. Observe that in order to understand the geometry of the moduli space around $\FF(\g)$ one can always replace $\g$ by the maximal subalgebra $H^0(X,\TT\FF(\g))\subseteq L$.

\begin{nteo} \label{teo1}  Let $\g\subseteq L$ be a maximal subalgebra of dimension $d$. Suppose further that $h^1(X, \TT\FF(\g))=0$. Then, there exist neighborhoods $\UU\subseteq \coprod_i S(d)_i$ and \emph{$\VV\subseteq \overline{\inv}$} of $\g$ and $\FF(\g)$ respectively such that the restriction $\phi: \UU\to \VV$ is an isomorphism. In particular, if $\g\subseteq L$ is rigid then $\FF(\g)$ is rigid in \emph{$\overline{\inv}$}.
\end{nteo}

As an immediate application, we get

\begin{ncorollary} Let $\g\subseteq L$ be a maximal semisimple subalgebra with $h^1(X,\TT\FF(\g))=0$. Then $\FF(\g)$ is rigid in \emph{$\overline{\inv}$}.
\end{ncorollary}

To our knowledge, the tangent sheaf of a generic element of every known stable family of $L$-foliations is trivial, i.e., the map $\g\otimes_\CC\OO_X\to \TT\FF(\g)$ is an isomorphism. When $X\subseteq \PP^r$ is arithmetically Cohen-Macaulay, Theorem \ref{teo1} can be stated in the following way.

\begin{ncorollary} \label{cor1} Let $X\subseteq \PP^r$ be smooth arithmetically Cohen-Macaulay of dimension $n$ and $\g\subseteq L$ a Lie subalgebra such that the morphism of sheaves $\g\otimes_\CC\OO_{X}\to\TT\FF(\g)$ admits a resolution of length $n-2$  of the form 
\[0 \rightarrow \bigoplus_{i=1}^{r_{n-2}}\OO_{X}(e^{n-2}_i)\rightarrow \cdots\rightarrow  \bigoplus_{i=1}^{r_{1}}\OO_{X}(e^{1}_i) \rightarrow \g\otimes_\CC\OO_{X}\rightarrow \TT\FF(\g)\rightarrow 0. \]
Then $\phi:\coprod_i S(d)_i\to \overline{\emph{\inv}}$ is an isomorphism locally around $\g$ and $\FF(\g)$. In particular, if $\g\subseteq L$ is rigid then $\FF(\g)$ is rigid in \emph{$\overline{\inv}$}. 
\end{ncorollary}

\begin{nobs} The above Corollary opens up the possibility of constructing stable families of $L$-foliations whose generic element has non-trivial tangent sheaf. We do so in Example \ref{adjsln}, where we show that the adjoint representation of the special linear algebra $\Sl_n$ induces a rigid element in $\inv$.
\end{nobs}

\begin{nobs}In the case of codimension $1$ foliations we also prove Theorem \ref{teoformas} and Corollary \ref{corformas}, which are the analogous statements corresponding to Theorem \ref{teo1} and Corollary \ref{cor1} for the classical spaces 
\[ \FF^1(X,\LL)= \{[\omega]\in \PP H^0(X,\Omega^1_X\otimes \LL) \hspace{0.1cm}\vert\hspace{0.1cm} \omega\wedge d\omega=0\} \]
parameterizing classes of integrable twisted $1$-forms with coefficients in some line bundle $\LL$ on $X$.
\end{nobs}

We also discuss the above results in terms of the complexes governating the deformation theories in play, namely the Chevalley-Eilenberg complex $C^\bullet(\g,L/\g)$ and the leaf complex $\LL_{\FF(\g)}$ of $\FF(\g)$.
\vspace{1cm}

\textbf{Acknowledgements.} The content of this article is part of the author's Ph.D thesis under the advice of Fernando Cukierman, to whom the author is very grateful for his support and guidance. Gratitude is also due to Carolina Araujo, César Massri, Jorge Vitório Pereira and Federico Quallbrunn for their useful comments and suggestions.

\section{Preliminaries}
Let us first recall some facts on (families of) coherent sheaves that will be needed in the rest of the article.

\begin{deff} Let $\FF$ be a coherent sheaf on $X$. The \emph{singular scheme} of $\FF$ is the closed subscheme 
\[ Sing(\FF)=Supp\left( \bigoplus_{i=1}^n\mathcal{E}xt^i (\FF,\OO_X)\right). \]
\end{deff}

\begin{deff} Let $f:Y\to S$ be a morphism of noetherian schemes. A \emph{family of coherent sheaves on the fibers of} $f$ is a coherent sheaf $\FF$ on $Y$. The family is \emph{flat} if $\FF$ is flat over $S$. A \emph{family of coherent sheaves on} $X$ with base $S$ is just a family of coherent sheaves on the fibers of the projection $X\times S\to S$. We will use the notation $\FF_s$ for the restriction of $\FF$ to the fiber $Y_s$. In the case $Y=X\times S$ we will just write $X_s$ instead.
\end{deff}

For the sake of clarity we will now state the so called ``Cohomology and Base Change'' Theorem.

\begin{teo}[Theorem 12.11, Chapter III, \cite{Hart}] \label{cybc} Let $f:Y\to S$ be a projective morphism of noetherian schemes and $\FF$ a flat family of coherent sheaves on its fibers. Let $s$ be a point in $S$. Then:
\begin{enumerate}
	\item if the natural map
\[ \varphi^i (s): R^if_*(\FF)\otimes k(s)\to H^i(Y_s,\FF_s) \]
is surjective, then it is an isomorphism and the same holds for every $s'$ in a sufficiently small neighborhood of $s$.
	\item if $\varphi^i(s)$ is surjective, then the following are equivalent:
		\begin{itemize}
			\item $\varphi^{i-1}(s)$ is also surjective.
			\item $R^i f_*(\FF)$ is locally free in a neighborhood of $s$.
		\end{itemize}
\end{enumerate}
\end{teo}

\begin{lema} \label{ext} Let $\FF$ be a flat family of coherent sheaves on $X$ with base $S$ and $s\in S$ be a point such that $h^1(X_s,\FF_s)=0$. Then for every global section $\X\in H^0(X_s,\FF_s)$ there exists a neighborhood $\UU\subseteq S$ of $s$ and an element $\X_\UU\in H^0(X\times \UU,\FF)$ whose restriction to $X_s$ coincides with $\X$.
\end{lema}
\begin{dem} By the semicontinuity theorem we know that $h^1(X_{s'},\FF_{s'})=0$ for every $s'$ in a neighborhood $\UU$ of $s$. Shrinking $\UU$ if necessary, by the first item in Theorem \ref{cybc} one must have $R^1 p_*(\FF)\otimes k(s') = 0$ for every $s'\in\UU$. But then $s\notin Supp(R^1p_*(\FF))$ and therefore 
\[ \varphi^0(s): p_*(\FF)\otimes k(s)\to H^0(X_s,\FF_s)\]
 is an isomorphism. In particular, every global section of $\FF_s$ extends to a section defined over some neighborhood $\UU$ of $s$.
\end{dem}

In the case $X\subseteq \PP^r$ is arithmetically Cohen-Macaulay of dimension $n$ we can replace the hypothesis on the first cohomology group of $\FF$ by a condition on the (graded) homological dimension of $\FF$. For every coherent sheaf $\FF$ on $X$ one can construct a graded free resolution 
\[ \cdots  \rightarrow \bigoplus_{i=1}^{k_{n}}\OO_{X}(-d_i^{n})\xrightarrow{d_n} \cdots\xrightarrow{d_1} \bigoplus_{i=1}^{k_0}\OO_{X}(-d_i^0)\xrightarrow{d_0} \FF \rightarrow 0. \]
We say that the resolution is of length $l$ if $k_l>0$ and $k_{l'}=0$ for $l'>l$. In the case $X=\PP^n$ for instance 
 Hilbert's Syzygy Theorem tells us that we can always construct a resolution of length $n$.

\begin{lema} \label{h1} Let $X\subseteq \PP^r$ be arithmetically Cohen-Macaulay of dimension $n$ and $\FF$ a coherent sheaf on $X$. If $\FF$ admits a free resolution of length $n-2$ then $h^1(X,\FF)=0$.
\end{lema}
\begin{dem} Recall that for an arithmetically Cohen-Macaulay $X$ the cohomology groups $H^i(X,\OO_X(d))$ vanish for every $d\in \ZZ$ and $1\leq i\leq n-1$. Now suppose a coherent sheaf $\FF$ admits a resolution of length $n-2$ as above. For every $0\leq r \leq n-3$ let us denote $ \KK^r=\ker(d_r)$. The first of these kernels fits in the exact sequence
\[ 0 \rightarrow \KK^0 \rightarrow \bigoplus_{i=1}^{k_0} \OO_{X}(-d_i^0) \xrightarrow{d_0} \FF \rightarrow 0 ,\]
which implies that $h^1(X,\FF)=h^2(X,\KK^0)$. In the same manner, for every $r$ the sequence
\[ 0 \rightarrow \KK^{r+1} \rightarrow \bigoplus_{i=1}^{k_{r+1}} \OO_{X}(-d_i^{r+1}) \xrightarrow{d_{r+1}} \KK^{r} \rightarrow 0 \]
is exact and therefore by looking at its associated sequence in cohomology we have $h^{j}(X,\KK^r)=h^{j+1}(X,\KK^{r+1})$ for every $j\leq n-2$. 
But then since $\KK^{n-3}\simeq \bigoplus_{i=1}^{k_{n-2}} \OO_{X}(-d_i^{n-2})$ we can conclude $h^1(X,\FF)=h^2(X,\KK^0)=\cdots=h^{n-1}(X,\KK^{n-3})=0.$
\end{dem}

\begin{lema} \label{resol} Let $X\subseteq \PP^r$ be a smooth arithmetically Cohen-Macaulay variety of dimension $n$, $\FF$ a coherent sheaf on $X$ and $V\otimes_\CC \OO_X\to \FF$ a surjection for some vector space $V$. If the morphism $V\otimes_\CC \OO_X\to \FF$ admits a resolution of the form 
\[0 \rightarrow \bigoplus_{i=1}^{r_{n-2}}\OO_{X}(e^{n-2}_i)\rightarrow \cdots\rightarrow  \bigoplus_{i=1}^{r_{1}}\OO_{X}(e^{1}_i) \rightarrow V\otimes_\CC\OO_{X}\rightarrow \FF\rightarrow 0, \]
then $h^1(X,\FF)=0$ and the corresponding map $V\to H^0(X,\FF)$ is surjective.
\end{lema}
\begin{dem} Let $V\otimes_\CC \OO_X\to \FF$ be as above and $\KK$ its kernel. Since both $\FF$ and $\KK$ admit a resolution of length $n-2$, by the previous Lemma we have $h^1(X,\FF)=h^1(X,\KK)=0$. Then looking at the cohomology of 
\[ 0 \rightarrow \KK \rightarrow  V\otimes_\CC\OO_X \rightarrow \FF\rightarrow 0\]
we can conclude that the map $H^0(X,V\otimes\OO_X)=V\to H^0(X,\FF)$ is surjective. 
\end{dem}

We end this section by stating a functorial characterization of the zero-locus of a global section $s\in H^0(X,\HH)$ for a locally free sheaf $\HH$.

\begin{deff} Let $\HH$ be a locally free sheaf on a scheme $X$ and $s:\OO_X \to \HH$ a global section. The zero-locus $Z(s)$ of $s$ is the closed subscheme whose ideal sheaf is the image of the dual map $s^\vee: \HH^\vee\to \OO_X$.
\end{deff}

\begin{lema}[Proposition 3.10, \cite{Quall}] \label{zeros} Let $\HH$ be a locally free sheaf on $X$ and $s\in H^0(X,\HH)$ a global section. Then the scheme $Z(s)$ represents the subfunctor $Z_s\hookrightarrow Hom(-,X)$ defined by
\[ Z_s(S)=\{f\in \Hom(S,X) \hspace{0.2cm}\vert\hspace{0.2cm} f^*s=0\in H^0(S,f^*\HH) \}.\]
\end{lema}

\subsection{Families of Lie algebras of vector fields}

In this section we will establish the notions regarding the deformation theory of Lie algebras that will be necessary for our purposes. Let us now first briefly recall the definition of Lie algebra cohomology. For further details on this subject the reader is referred to [Section 7.7, \cite{Wei}]. 

Let $\g$ be a finite dimensional complex Lie algebra and $M$ a $\g$-module. The cohomology groups  $H^\bullet(\g,M)$ are defined as the homology of the Chevalley-Elilenberg complex 
$(C^\bullet(\g,M),\delta) $
of $M$, where $\C^\bullet(\g,M)=\Hom_\CC(\bigwedge^\bullet \g, M)$ and the differential $\delta^k: C^{k}(\g,M)\to C^{k+1}(\g,M)$ is given by the equations
\[ \delta(f)(x_1\wedge\cdots\wedge x_{k+1})=\sum_{i=1}^{k+1}(-1)^{i} x_i\cdot f(\widehat{x}_i) 
+\sum_{i<j}(-1)^{i+j}f([x_i,x_j]\wedge \widehat{x}_{ij}). \] 
We will use the notation $Z^k(\g,M)=\ker(\delta^k)$ and $B^k(\g,M)=\im(\delta^{k-1})$.

\begin{deff} Let $S$ be a $\CC$-scheme. A \emph{family of Lie algebras of dimension $r$} with base $S$ is a locally free sheaf $\OO_S$-module of rank $r$ together with a morphism  $[-,-]:\bigwedge^2\LL\to \LL$ that satisfies the Jacobi identity at the level of local sections.
\end{deff}

Now let $L=H^0(X,\TT X)$ be the complex Lie algebra of global vector fields on $X$. For a base $S$, let us consider the \emph{trivial} family over $S$ with fiber $L$, which consists of the sheaf $L_S:= L\otimes_\CC \OO_S$ together with the $\OO_S$-linear extension of the Lie bracket on $L$. 

\begin{deff} A \emph{family of subalgebras of dimension $d$ of $L$} with base $S$ is a locally free subsheaf $\GG \hookrightarrow L_S$  such that $rg(\GG)=d$ and $[\GG,\GG]\subseteq \GG$. We will use the notation $\mathfrak{Sub}_d: Sch_\CC \to Sets$ for the functor assigning to each scheme $S$ the set $\mathfrak{Sub}_d(S)$ of families of subalgebras of $L$ of dimension $d$ over $S$.
\end{deff}

\begin{ej}  \label{familia1} For every $n\geq 5$ consider the family $\GG$ of subalgebras of vector fields in $\PP^n$ of dimension $3$ with base $S=Spec(\CC[t])$ generated by the elements
\begin{align*} \X_1&= (x_1+tx_2)\frac{\partial}{\partial x_1} + x_2 \frac{\partial}{\partial x_2}+ (x_4+tx_5)\frac{\partial}{\partial x_4} + x_5 \frac{\partial}{\partial x_5}, \\
\X_2 &=-x_0\frac{\partial}{\partial x_1} -x_3\frac{\partial}{\partial x_4}\mbox{\hspace{0.2cm} and} \\
\X_3 &= -t x_0 \frac{\partial}{\partial x_1}-x_0 \frac{\partial}{\partial x_2} - t x_3 \frac{\partial}{\partial x_4}-x_3 \frac{\partial}{\partial x_5}
\end{align*}
As a family of Lie algebras, its structure is given by 
\[[\X_1,\X_2]=\X_2,  \hspace{0.3cm} [\X_1,\X_3]=t\X_2 + \X_3 .\] 
For $t\neq 0$  the Lie algebra $\GG(t)$ is isomorphic to $\GG(1)$ (via the map defined by $\X_2(t)\mapsto t\X_2(t)$).
For $t=0$ on the other hand the algebra $\GG(0)$ is not isomorphic to $\GG(1)$ (the complete classification of complex Lie algebras of dimension less or equal than $4$  can be found in \cite{class4d}). 
\end{ej}

\begin{obs} \label{brack} The functor $\mathfrak{Sub}_d$ is a subfunctor of the Grassmann functor: every family of subalgebras of dimension $d$ is in particular a family of subspaces of $L$. More precisely, let
\[0\rightarrow \KK \rightarrow L_{Gr(L,d)} \xrightarrow{q} \Q \rightarrow 0 \]
be the tautological exact sequence over the Grassmannian of $d$-dimensional subspaces of $L$ and $b: \bigwedge^2 L_{Gr(L,d)}\to L_{Gr(L,d)}$ the Lie bracket on $L_{Gr(L,d)}$. We will also denote by $b$ its restriction $b:\bigwedge^2 \KK \to L_{Gr(L,d)}$. Then for a scheme $S$ every family of subalgebras $\GG\subseteq L_S\in \mathfrak{Sub}_d(S)$ is of the form $f^*(\KK\subseteq L_{Gr(L,d)})$ for some unique $f: S\to Gr(L,d)$. Observe that the condition $[\GG,\GG]\subseteq \GG$ is equivalent to the vanishing of the map of vector bundles
$$f^*(q\circ b): \bigwedge^2 \GG \to f^*\Q\simeq L_S/\GG.$$
\end{obs}

\begin{prop} \label{lierep}The functor $\mathfrak{Sub}_d$ is represented by a closed  subscheme $S(d)\subseteq Gr(L,d)$.
\end{prop}
\begin{proof} Following Remark \ref{brack}, let us denote by $s$ the global section of the locally free sheaf $\HH om(\bigwedge^2\KK,\Q)$ corresponding to the morphism 
$$q\circ b: \bigwedge^2 \KK \to \Q.$$
 Let $S$ be a scheme. The discussion above implies that the functor $\mathfrak{Sub}_d$ can actually be described as 
\[ \mathfrak{Sub}_d(S)=\{f\in \Hom(S,Gr(L,d)) \hspace{0.2cm}\vert\hspace{0.2cm} f^*s=0\in H^0(S,f^*\HH om(\bigwedge^2\KK,\Q)) \}.\]
By Lemma \ref{zeros} this implies that $\mathfrak{Sub}_d$ is represented by the closed subscheme $Z(s)=:S(d)\subseteq Gr(L,d)$ of the section $s$.
\end{proof}

\begin{obs} \label{obslie}The automorphism group $\Aut(X)$ acts on $\TT X$ and therefore on $L$. This yields an action on $G(L,d)$,  which of course stabilizes $S(d)$. 
\end{obs}

\begin{deff} A subalgebra $\g\subseteq L$ is \emph{rigid} if its orbit $\Aut(X)\cdot \g\subseteq S(d)$ contains a non-empty Zariski open subset.
\end{deff}

 The following proposition allows us to check rigidity in terms of the Chevalley Eilenberg complex of the $\g$-module $L/\g$.

\begin{prop}\label{tangentelie} Let $\g\subseteq L$ be a subalgebra of dimension $d$. Then
\begin{enumerate}
	\item the tangent space to $S(d)$ at $\g$ is isomorphic to $Z^1(\g,L/\g)$ and
	\item under the above isomorphism, the tangent space to the orbit $\emph{\Aut(X)}\cdot \g$ at $\g$ corresponds to $B^1(\g,L/\g)$.
\end{enumerate}
\end{prop}
\begin{dem} This is [Proposition 6.1, \cite{Rich}] and [Proposition 7.1, \cite{Rich}].
\end{dem}

\begin{ej} Let $\g$ be a semisimple subalgebra of $L$. By a theorem of Whitehead we know that $H^1(\g,M)=0$ for every finitely generated $\g$-module $M$. In particular, $\g$ is rigid in $L$.
\end{ej}

\begin{ej}\label{calcfamilia1} Let us consider the Lie subalgebra $\g\subseteq H^0(\PP^n,\TT \PP^n)$ of dimension $3$ corresponding to the fiber over $t=1$ in Example \ref{familia1}. Using the software GAP (see \cite{GAP}) we can easily compute the corresponding tangent spaces and obtain $\dim_\CC(Z^1(\g,L/\g))=32$ and $\dim_\CC(B^1(\g,L/\g))=28$. 
\end{ej}

\subsection{Foliations and moduli spaces}
We will now give the general definitions and properties on foliations that we will be used afterwards.

\begin{deff} A singular foliation $\FF$ of codimension $q$ on $X$ is a short exact sequence 
$$0\rightarrow \TT\FF \rightarrow \TT X \rightarrow N_\FF \rightarrow 0, $$
such that the tangent sheaf $\TT \FF$ of $\FF$ is closed under the Lie bracket of vector fields and the normal sheaf $N_\FF$ is torsion-free of rank $q$. The singular scheme $Sing(\FF)$ of the foliation is the singular scheme of its normal sheaf.
\end{deff}

\begin{deff} A family of foliations of codimension $q$ on $X$ with base $S$ is a short exact sequence 
$$0\rightarrow \TT\FF \rightarrow \TT X_S \rightarrow N_\FF \rightarrow 0 $$
such that $\TT\FF$ is closed under the Lie bracket of relative vector fields and its normal sheaf $N_\FF$ is torsion-free on the fibers and of rank $q$. Its singular scheme $Sing(\FF)$ is the singular scheme of its normal sheaf. The family is said to be \emph{flat} if $N_\FF$ is flat over $S$. 
\end{deff}

\begin{deff} For every polynomial $P\in \QQ[t]$ let $\INV^P: Sch\to Sets$ be the functor that assigns to every scheme $S$ the set of flat families $\FF$ of foliations on $X$ with base $S$ such that the restriction of $N_\FF$ to every fiber has Hilbert polynomial equal to $P$.
\end{deff}

With a similar argument to the proof of Proposition \ref{lierep} one can show the following

\begin{prop} [Proposition 6.3, \cite{Quall}] The functor $\INV^P$ is represented by a locally closed subscheme \emph{$\inv^P\subseteq \Quot^P(\TT X)$}.
\end{prop}

The following definition is in order to mantain a clearer exposition.

\begin{deff} \label{definv}We will use the notation 
$$\inv= \coprod_{P\in \QQ[t]} \inv^P. $$
We will also denote by $\overline{\inv}\subseteq \Quot(\TT X)$ its closure, i.e., the scheme representing the functor $\overline{\INV}$ of flat families of coherent involutive subsheaves of $\TT X$ (this is, families whose normal sheaf may not be torsion-free). Of course, these spaces are equipped with a universal family which will be denoted by $\FF_\inv$.
\end{deff}

\begin{obs} Just as in Remark \ref{obslie}, the natural action of $\Aut(X)$ on $\Quot(\TT X)$ restricts to an action on $\overline{\inv}$. We will say that $\FF$ is rigid if its orbit is Zariski open.
\end{obs}

Let us now move on to describing the tangent space to $\inv$ at a point $\FF$.

\begin{deff} Let $\FF$ be a foliation on $X$. We say that a global vector field $\X\in H^0(X,\TT X)$ is an \emph{infinitesimal automorphism of} $\FF$ if the flow associated to $\X$ satisfies $\exp(t\X)^*\FF=\FF$. We will denote by $\aut(\FF)\subseteq H^0(X,\TT X)$ the vector space of infinitesimal automorphisms of $\FF$.
\end{deff}

\begin{deff} Let $\FF$ be a foliation on $X$. The \emph{leaf complex} of $\FF$ is the complex
\[L_\FF: \TT X\xrightarrow{d_0} Hom(\TT\FF,N_\FF)\xrightarrow{d_1} Hom(\bigwedge^2 \TT\FF,N_\FF)\xrightarrow{d_2} \cdots, \]
where the morphisms are defined at the level of local sections by
\begin{align*} d_0(v)(w)&=\overline{[v,w]} \text{ and} \\ 
 d_r(\alpha)(w_1,\dots,w_{r+1})&=  \sum_{i=1}^{r+1} (-1)^i \overline{[w_i,\alpha(\widehat{w_i})]} + \sum_{i<j}(-1)^{i+j}\alpha([w_i,w_j],\widehat{w_{ij})}.   
 \end{align*}
\end{deff} 

This object is very relevant for the theory of deformations of $\FF$. Recall that the tangent space to $\Quot(\TT X)$ at the point $\FF$ is isomorphic to  $\Hom(\TT\FF,N_\FF)$. Let $\psi:\Hom(\TT\FF,N_\FF)\to \TT_\FF \Quot(\TT X)$ be the canonical isomorphism. Using the ideas presented in \cite{GM} we can easily relate the leaf complex of $\FF$ to the tangent space to $\inv$ at the point $\FF$.
 
\begin{teo}\label{tangente} Let $\FF$ be a foliation on $X$, $\OO(\FF)$ its orbit under the natural action of $\Aut(X)$ on \emph{$\inv$} and $L\to \TT_\FF \OO(\FF)$ the derivative of the action at the point $\FF$. The following diagram of vector spaces is commutative and has exact rows:
\[\begin{tikzcd}
 0 \arrow[r] & \aut(\FF) \arrow[r] & L \arrow[d, "H^0(d_0)"]\arrow[r] & \TT_\FF \OO(\FF) \ \arrow[d, hookrightarrow]\arrow[r] & 0 \\
 & 0\arrow[r] &  \ker(H^0(d_1)) \arrow[r, "\psi" ] & \TT_\FF \emph{\inv} \arrow[r] & 0.
 \end{tikzcd}\]

\end{teo}
\begin{dem} The deformation problem considered in \cite{GM} includes possible deformations of $X$. There, the author proves that the vector space consisting of first order deformations of $\FF$ coincides with the first hypercohomology group  $\mathbb{H}^1(X,L_\FF)$ of the leaf complex (see \cite[Theorem 2.4]{GM}. This is, if $\UU$ is a sufficiently fine cover of $X$ then giving a first order deformation $\mathcal{X}$ of $X$ and an involutive subsheaf $\FF'\subseteq \TT \mathcal{X}_{Spec(D)}$ is equivalent to specifying a $1$-cocycle $\theta\in C^1(\UU,\TT X)$ (which defines $\mathcal{X}$) and an element $\eta\in C^0(\UU,Hom(\TT\FF, N_\FF))$ satisfying some equations of compatibility and integrability.

In order to apply this correspondence to our deformation problem we have to consider elements of the form
$$(0,\eta) \in C^1(\UU,\TT X)\oplus C^0(\UU, Hom(\TT\FF, N_\FF)).$$
In this case, the compatibility equations indicate that the element $\eta$ defines a global section $\eta\in H^0(X,Hom(\TT\FF,N_\FF))$. The integrability condition on the other hand is exactly $d_1(\eta)=0$.

Alternatively, one can apply \cite[Theorem 1.4]{GM} to the case where $\mathcal{X}$ is the trivial deformation of $X$. 
\end{dem}

Following \cite{deMed} every foliation $\FF$ on $X$ can be defined by an \emph{integrable} and \emph{locally decomposable away from its singular points} differential $q$-form as follows: for every foliation $\FF$ of codimension $q$ on $X$ we can consider the corresponding monomorphism $\det(N_\FF)^\vee\hookrightarrow \Omega^q_X$, which is induced by a unique element 
\[ \omega_\FF\in H^0(X,\Omega^q_X\otimes \det(N_\FF)) \]
up to scalar multiplication. Thus $\omega_\FF$ satisfies the following properties:
\begin{enumerate}
	\item around every point $p$ such that $\omega_\FF(p)\neq 0$ there exist local $1$-forms $\omega_i$ such that $\omega_\FF=\omega_1\wedge\dots\wedge\omega_q$ ( i.e., $\omega_\FF$ \emph{is decomposable around} $p$), 
	\item $\ker(\omega_\FF)=\TT \FF$ is involutive ($\omega_\FF$ \emph{is integrable}) and
	\item $Z(\omega_\FF)=Sing(\FF)$ is of codimension at least $2$.
\end{enumerate}
The first condition is equivalent to requiring that $\omega_\FF$ satisfies the Plücker equations
\[\tag{P} \imath_v(\omega_\FF)\wedge\omega_\FF=0 \]
for every local section $v$ of $\bigwedge^{q-1}\TT X$. Integrability on the other hand corresponds to the Frobenius integrability condition 
\[\tag{F} \imath_v(\omega_\FF)\wedge d\omega_\FF=0 \]
for every local section $v$ of $\bigwedge^{q-1}\TT X$.
Also, the zero locus of $\omega_\FF$ coincides with the singular scheme of $\FF$.

We will be specially interested in the following type of singularities.
\begin{deff} Let $\FF$ be a codimension $1$ foliation on $X$. We say that $p\in X$ is a \emph{Kupka singularity} if $\omega_\FF(p)=0$ and $d\omega_\FF(p)\neq 0$. We will denote by $\KK(\FF)$ the set of Kupka points of $\FF$. 
\end{deff}

\begin{obs} Around a point $p\in \KK(\FF)$ the foliation $\FF$ can be described as a pullback of a germ of foliation at $(\CC^2,0)$ with an isolated singularity at the origin. In particular, $\KK(\FF)$ is pure of codimension $2$. For further details on the local structure of $\FF$ around a Kupka point see  [Capitulo 1.4, \cite{LN}].
\end{obs}

\begin{deff} \label{reeb} Let $\FF$ be a codimension $1$ foliation. We say that a singularity $p\in X$ is of \emph{Reeb-type} if there exists an anallytic neighborhood $\UU$ of $p$ such that $\FF\vert_\UU$ is the foliation induced by some $\omega=\sum_{i=1}^n f_i dz_i$, where $f_i(p)=0$ for every $i$ and $(df_1)_p,\dots,(df_n)_p$ are linearly independent.
\end{deff}

Conversely, for a line bundle $\LL$ on $X$ every element $\omega\in H^0(X,\Omega^q_X\otimes \LL)$ satisfying conditions (P), (F) and $\codim(Z(\omega))\geq 2$ defines a singular foliation $\FF_\omega$ of codimension $q$ with tangent sheaf 
\[\TT\FF_\omega=\ker(\omega) \]
and such that $\det(N_\FF)=\LL$.
Moreover, two such forms induce the same foliation if and only if they differ by a multiplicative constant $\lambda\in H^0(X,\OO_X^*)=\CC$. This motivates the following:

\begin{deff} Let $\LL$ be a line bundle on $X$. The space of foliations of codimension $q$ and degree $\LL$ on $X$ is defined as 
\[ \FF^q(X,\LL)=\{[\omega]\in \PP H^0(X,\Omega^q_X\otimes \LL) \hspace{0.1cm}|\hspace{0.1cm} \omega\text{ satisfies $\codim(Z(\omega))\geq 2$, (P), (F)} \} .\]
\end{deff}

In the case where $Pic(X)$ is discrete, this scheme also admits a natural action of $\Aut^0(X)$. A point $\FF\in \FF^q(X,\LL)$ is said to be rigid if its orbit contains a non-empty Zariski open subset. 

\begin{obs} \label{familiataut}Using the restriction of the tautological sequence on $\PP H^0(X,\Omega^q_X\otimes \LL)$ we can construct the tautological family of foliations $\FF^q_\LL$ with base $\FF^q(X,\LL)$, i.e., the family whose restriction to the fiber over each $[\omega]$ is the foliation $\FF_\omega$. 
\end{obs}

We have now defined two spaces paremetrizing foliations on $X$ (with different notions of continuity). In the case $Pic(X)$ is discrete, we can indeed compare their geometries.

\begin{prop} \label{compararinv}Suppose that $Pic(X)$ is discrete and let $\FF_0$ be a codimension $q$ foliation on $X$. Let $\LL=\det(N_{\FF_0})$ and $\C\subseteq \FF^q(X,\LL)$ be the stratum of the flattening stratification of $N_{\FF^q_\LL}$ passing through $\FF_0$. Then there exists a neighborhood $\UU\subseteq \emph{\inv}$ of $\FF$ such that the morphism
$\omega: \UU\to \C $
defined as  $\omega(\FF')= [\omega_{\FF'}]$
is an \emph{$\Aut^0(X)$}-equivariant isomorphism around $\FF_0$ and $[\omega_{\FF_0}]$.
\end{prop}
\begin{proof} Let $\FF$ be a flat family of codimension $q$ foliations with affine base $S=Spec(A)$ and $s\in S$ some point such that $\FF_s=\FF_0$. The relative version of the argument above yields a unique (up to multipllication by units of $A$) global section
\[ \omega_\FF\in H^0(X\times S, \Omega^q_{X\times S|S} \otimes \det(N_\FF)) \]
such that the kernel of the contraction 
\[ \omega_\FF: \TT X_S\to \Omega^{q-1}_{X\times S|S} \otimes \det(N_\FF)\] 
coincides with $\TT \FF$. Its image is canonically isomorphic to $N_\FF$ (which is flat), which implies that for every $s\in S$ we have 
\[ \tag{$\star$} \ker(\omega_\FF(s))=\ker(\omega_\FF)_s=\TT \FF_s .\]
Also, being $Pic(X)$ discrete by the Seesaw Theorem [Corollary 6, p.54, \cite{AbV}] we know that $\det(N_\FF)=p_1^*\det(N_{\FF_0})$ and  therefore 
\[ \Omega^q_{X\times S|S}\otimes \det(N_\FF)=p_1^*(\Omega^q_X\otimes \LL).\]  
This is, the section $\omega_\FF$ defines an element of $H^0(S,H^0(X,\Omega^q_X\otimes \LL)\otimes_\CC \OO_S)$ or equivalently a morphism $S\to H^0(\Omega^q_X\otimes \LL)$. By $(\star)$ (and the fact that $\omega_\FF$ is integrable and locally decomposable away from its singular points), this descends to a morphism 
$$ \omega_\FF: S\to \FF^q(X,\LL),$$
uniquely determined by $\FF$ and such that the family $\FF$ is the pullback under $\omega_\FF$ of the tautological family over $\FF^q(X,\LL)$. Since $\FF$ is indeed a flat family $\omega_\FF$ must factor through the flattening stratification associated to $N_{\FF^q_\LL}$. 

Now consider the Zariski open set $\UU'=\{ \FF\in \inv  \hspace{0.1cm}|\hspace{0.1cm} \det(N_\FF)=\LL \}$ and let $\UU\subseteq \UU'$ be the connected component passing through $\FF_0$, which is of course $\Aut^0(X)$-invariant. By applying the argument above to the restriction of $\FF_\inv$ to an affine cover of $\UU$ we get an $\Aut^0$-equivariant morphism $\omega=\omega_{\FF_\inv}:\UU\to\C$ satisfying $\omega(\FF')=[\omega_{\FF'}]$ for every $\FF'\in \UU$ and such that the family $\FF_\inv\vert_\UU$ is the pullback under $\omega$ of $\FF^q_\LL\vert_\C$. In particular, the universal family over  $S=Spec(\OO_{\inv,\FF_0})$  is the pullback under $\omega$  of the tautological family over $\C$, which by the representability of $\INV$ easily implies that $\omega$ induces an isomorphism locally around $\FF_0$ and $[\omega_{\FF_0}]$.
\end{proof}

\section{L-foliations}

Let $G\subseteq \Aut(X)$ be a connected Lie group with Lie algebra $\g$. Let us suppose that the natural action of $G$ on $X$ satisfies that $\dim G\cdot x$ is constant in codimension $1$. The leaves of the foliation $\FF(\g)$ induced by this action are exactly the orbits of maximal dimension. The derivative of $G\hookrightarrow \Aut(X)$ at the identity corresponds to the inclusion $\g\hookrightarrow H^0(X,\TT X)= L$. The restriction of global sections yields a morphism 
$$ \rho_\g: \g\otimes_\CC \OO_X \to \TT X$$
whose image is the tangent sheaf $\TT\FF(\g)$. As for the twisted differential form defining $\FF(\g)$, we will use the notation $\omega(\g)=\omega_{\FF(\g)}$.

\begin{deff} Let $\FF$ be a foliation on $X$. We say that $\FF$ is an \emph{L-foliation} if $\FF=\FF(\g)$ for some subalgebra  $\g\subseteq L$.
\end{deff}
As we will now see, these foliations admit some additional structures.

\begin{deff} Let $\FF$ be a foliation on $X$. We say that a global vector field $\X\in H^0(X,\TT X)$ is a symmetry of $\FF$ if $\X\in \aut(\FF)$ and $\X$ is not tangent to $\FF$. 
\end{deff}

It is worth mentioning that a general foliation does not admit neither symmetries nor algebraic leaves. The following is an adaptation of [Théoreme 1.8, \cite{CD}]. 

\begin{teo} \label{estructura} Let $\FF$ be an $L$-foliation on $X$. Then
\begin{enumerate}
	\item $\FF$ admits a symmetry, or
	\item every leaf of $\FF$ is algebraic.
\end{enumerate}
\end{teo}
\begin{dem} If $\FF$ is an $L$-foliation, then $\FF=\FF(\g)$ for $\g:=H^0(X,\TT\FF)\hookrightarrow L$. Let $\Aut(X)^0$ be the connected component of the identity and $G_1$ be the connected subgroup
\[ G_1 := \langle exp(tY) \rangle_{Y\in \g} < \Aut(X) \]
generated by the flows of $\g$. 
On the other hand, let us consider the group $\Aut(\FF)^0= \{ \phi \in \Aut^0(X) \vert \phi^*\FF=\FF \}$ consisting of automorphisms of $\FF$.
Observe that $\Aut(\FF)^0$ is the fiber over $\omega_\FF$ of the map $\omega: \Aut(X)^0\to \PP H^0(X,\det(N_\FF))$ satisfying $\omega(\phi)=[\omega_{\phi\cdot \FF}]$. In particular, $\Aut(\FF)^0$ is an algebraic subgroup of $\Aut(X)^0$ containing $G_1$. Let $\ov{G_1}$ be the Zariski closure of $G_1$. Then 
$$G_1<\ov{G_1} < \Aut(\FF)^0.$$
If $G_1=\ov{G_1}$, then the orbit  $G_1\cdot x$ is algebraic for every $x\in X$. But then every leaf of $\FF$ is algebraic.

If on the other hand the above inclusion were strict, there must be an element
 $\X \in T_e \Aut(\FF)^0\setminus \g$, i.e. a symmetry of $\FF$.
\end{dem}

\begin{ej} \label{jinv}For $\g=sl_2$, let us consider the foliation $\FF(\g)$ in $X=\PP Sym^4(\CC^2)\simeq \PP^4$ associated  to the action of $\PP GL_2(\CC)$ by changes of coordinates, which was studied in \cite{notej}. This foliation satisfies condition 2 in the above Theorem, since the closure of a generic orbit coincides with a fiber of the $j$-invariant. 

The restriction of this foliation to an appropiate hyperplane yields an important construction in the theory of foliations in projective spaces, namely the \emph{exceptional component}: this restriction is induced by an action of the affine Lie algebra $\aff(\CC)$ and is rigid in $\FF^1(\PP^3,4)$. 
\end{ej}

\begin{ej} More generally, by [Théoreme 1.22,  \cite{CD}] every $L$-foliation $\FF(\g)$ on $\PP^n$ of codimension $1$ such that $[\g,\g]=\g$ satisfies the second condition in Theorem \ref{estructura}. 

In \cite{CP} the authors study foliations on $\PP^n$ induced by infinitesimal actions of Lie subalgebras $\g\subseteq H^0(\PP^n,\TT\PP^n)$ which are locally free in codimension one, i.e., such that the morphism $\g\otimes_\CC \OO_X\to \TT\FF(\g)$ is an isomorphism. Regarding their stability, in [Corollary 6.1, \cite{CP}] they show that under some hypotheses on the singular scheme of $\FF(\g)$, the elements $\g$ and $\FF(\g)$ have isomorphic neighborhoods in $S(d)$ and $\FF^q(\PP^n,n+1)$ respectively. As a consequence, they are able to construct irreducible components of this space consisting generically of $L$-foliations.
\end{ej}

\subsection{Families of subalgebras and foliations}
The last example shows that under certain hypotheses it is possible to define a morphism between $S(d)$ and some moduli space of foliations. We will now explain the formalisms that lead to a similar constrution around every point in the moduli space of subalgebras.

\begin{deff} Let $\GG\subseteq L_S$ be a family of Lie subalgebras parametrized by some scheme $S$. The associated family $\FF(\GG)$ of  involutive distributions is the family whose tangent sheaf is the image of the morphism
$$ \rho_{\GG}: \GG \otimes \OO_{X\times S} \to \TT X_S.$$
\end{deff}

Since the scheme $S(d)$ is equipped with a universal family $\GG_{S(d)}\subseteq L_{S(d)}$ of subalgebras of $L$ of dimension $d$, it is only natural to consider the case $S=S(d)$. We will use the notation $\FF_d:=\FF(\GG_{S(d)})$ for the induced family
\begin{equation*} \FF_d: 0 \rightarrow \TT\FF_d\rightarrow \TT X_{S(d)} \rightarrow N_d\rightarrow 0. 
\end{equation*}
Being this family not necessarily flat, we have to consider the flattening stratification
$$\coprod_i S(d)_i \to S(d)$$
associated to the sheaf $N_d$. We thus get over each stratum $S(d)_i$ a flat family $\FF_i:=\FF(\GG_{S(d)_i})$ and therefore a morphism $\phi_i: S(d)_i\to \ov{\inv}$ such that 
$$0 \rightarrow \TT\FF_i\rightarrow \TT X_{S(d)_i} \rightarrow N_d\rightarrow 0 $$
is the pullback under $\phi_i$ of the universal family over $\overline{\inv}$. Observe that if $\g\in S(d)_i$ then its orbit $\OO(\g)$ under the action of $\Aut(X)$ is contained in $S(d)_i$. It follows from the above that $\phi_i$ is $\Aut(X)$-equivariant. In particular, we have $\phi_i(\OO(\g))=\OO(\FF(\g))$.

\begin{deff} We will denote $\phi:\coprod_i S(d)_i\to \overline{\inv}$ the morphism induced by the  $\phi_i$'s. 
 \end{deff}
 
For a family $\GG\subseteq L_S$ of subalgebras over some scheme $S$ we have $\GG\subseteq H^0(S,p_{2,*}\TT\FF(\GG))$. In general, this may be a strict inclusion. In what follows, we want to be able to recover the family of subalgebras from its associated family of foliations.  The following definition was first introduced in \cite{CD}:

\begin{deff}Let $\g\subseteq L$ be a subalgebra. We say that $\g$ is \emph{maximal} if $\g=H^0(X,\TT\FF(\g))$.
\end{deff}

 \begin{obs} \label{secciones} Maximality defines an open subscheme of $S(d)$. Indeed, a point $\g\in S(d)$ is maximal if and only if the fiber $(p_{2,*}\TT\FF_d)(\g)$ has dimension less than $d+1$. These are also the points around which the quotient sheaf $p_{2,*}(\TT\FF_d)/\GG_{S_d}$ is zero. Observe also that in order to understand the geometry of $\inv$ around $\FF(\g)$ we can replace $\g$ by $H^0(X,\TT\FF(\g))$.
 \end{obs}

\begin{ej} \label{familia1c} Let us consider the family $\GG$ on $\PP^n$ with base $S=Spec(\CC[t])$ introduced in Example \ref{familia1}. The family $\FF=\FF(\GG)$ is of dimension $3$ and generated by the elements $\X_1,\X_2,\X_3\in H^0(\PP^n\times S, \TT\PP^n_S)$, which induce the morphism 
\[ \X: \OO_{\PP^n\times S}^{\oplus 3} \xrightarrow{\sim} \TT\FF .\]
Since $\X(t):\OO_{\PP^n}^{\oplus 3}\to \TT\FF_t$ is an isomorphism for every $t\in \CC$, the family $\FF(\GG)$ is automatically flat over $S$.
\end{ej}

\subsection{Stability}
We will now prove the results stated in the introduction.
Recall that according to Remark \ref{secciones}, in order to study the geometry of $\ov{\inv}$ around $\FF(\g)$ there is no loss of generality in assuming that $\g$ is maximal.

\begin{teo} \label{isolocal} Let $\g\subseteq L$ be a maximal subalgebra of dimension $d$. Suppose further that $h^1(X, \TT\FF(\g))=0$. Then, there exist neighborhoods $\UU\subseteq \coprod_i S(d)_i$ and \emph{$\VV\subseteq \overline{\inv}$} of $\g$ and $\FF(\g)$ of respectively such that the restriction $\phi: \UU\to \VV$ is an isomorphism. In particular, if $\g\subseteq L$ is rigid then $\FF(\g)$ is rigid in \emph{$\overline{\inv}$}.
\end{teo}
\begin{dem} 
Let $0\rightarrow \TT\FF \rightarrow \TT X_{\overline{\inv}}\rightarrow N_\FF\rightarrow 0$ be the universal family over $\overline{\inv}$ and $\XX_0:\OO_{X}^{\oplus d} \twoheadrightarrow \TT \FF(\g)$ the morphism induced by the action of $\g$. By Lemma \ref{ext} we can extend this morphism, i.e., there exists some neighborhood $\VV$ of $\FF(\g)$ and a morphism  $\XX_\VV:\OO_{\VV}^{\oplus d}\to \TT \FF\vert_\VV$ restricting to $\XX_0$. Without loss of generality we can assume that $\XX_\VV$ is also surjective.

We will denote by $\pi_1$ and $\pi_2$ the projections of $X\times\VV$ onto $X$ and $\VV$ respectively. By Theorem \ref{cybc} (shrinking $\VV$ if necessary) we know that $\pi_{2,*}\TT\FF\vert_\VV$ is locally free. The pushforward of $\TT\FF\hookrightarrow  \TT X_\VV$ therefore becomes a monomorphism of vector bundles
$$ \pi_{2,*}\TT\FF\vert_\VV \hookrightarrow  \pi_{2,*}\TT X_\VV=H^0(X,\TT X)\otimes \OO_\VV=L_\VV,$$
whose image is an element of $\mathfrak{Sub}_{d}(\VV)$. By the representability of this functor there exists a morphism $f:\VV\to S(d)$ such that $\pi_{2,*}\TT\FF\vert_\VV$ is the pullback under $f$ of the universal subalgebra. Let us now consider the commutative diagram 
\[
\xymatrix{
X\times \VV \ar[d]^{\pi_2} \ar[r]^{1\times f} & X\times{S(d)}  \ar[d]^{p_2}\\
\VV \ar[r]^f & S(d), }
\]
where $p_2: X\times S(d)\to S(d)$ is the usual projection.
Being $\TT\FF_d$ globally generated, as a subsheaf of $\TT X_{S(d)}$ it must coincide with the image of the restriction of sections $p_2^* p_{2,*}\TT\FF_d\to \TT X_{S(d)}$. Since $\g$ is maximal, by Remark \ref{secciones} we can restrict ourselves to a neighborhood $\wt{\UU}\subseteq S(d)$ of $\g$ and redefine $\VV=f^{-1}(\wt{\UU})$ in order to guarantee the equality
\[ p_{2,*}(\TT\FF_{\wt{\UU}}\hookrightarrow \TT X_{\wt{\UU}})=\GG_{\wt{\UU}}\hookrightarrow L_{\wt{\UU}}, \]
where $\GG_{\wt{\UU}}=\GG_{S(d)}\vert_{\wt{\UU}}$. But then the tangent sheaf of the family $f^*(\FF_d\vert_{\wt{\UU}})$ is
\begin{align*} \im \left((1\times f)^*(p_2^* p_{2,*}(\TT\FF_d\vert_{\wt{\UU}})\to \TT X_{\wt{\UU}})\right)& =  \im\left( (f \circ \pi_2)^* p_{2,*}\TT\FF_{\wt{\UU}} \to \TT X_\VV \right)  \\
&= \im\left( (f \circ \pi_2)^* \GG_{\wt{\UU}} \to \TT X_\VV \right) \\
&= \im \left( \pi_2^*\pi_{2,*}\TT\FF\vert_\VV \to \TT X_\VV\right).
\end{align*}
Being $\TT\FF\vert_\VV$ also globally generated, we have 
\[ \im ( \pi_2^*\pi_{2,*}\TT\FF\vert_\VV \to \TT X_\VV)=\TT\FF\vert_\VV. \]
This is, the pullback of $\FF_d\vert_{\wt{\UU}}$ is 
$$0 \rightarrow \TT\FF\vert_\VV \rightarrow \TT X_\VV \rightarrow (1\times f)^*N(d)\rightarrow 0  $$
and therefore $f^*N(d)\simeq N_\FF$ is flat over $\VV$. Then, $f$ factorizes over the flattening stratification
$$ \VV\rightarrow \coprod_k S(d)_k \rightarrow S(d) $$
and therefore  (shrinking $\VV$ if necessary) over the stratum $S(d)_i$ passing through $\g$. Let us also observe that we have proven the equality 
\[ f^*(\FF_d\vert_{\wt{\UU}})=\FF_{\overline{\inv}}\vert_\VV. \]

Let $\UU=\phi_i^{-1}(\VV)\subseteq S(d)_i$. The composition $\phi_i \circ f:\VV\to \inv$ satisfies $(f\circ \phi_i)^*\FF\vert_\VV=\FF\vert_\VV$. It follows from the representability of $\overline{\INV}$ that this composition coincides with the inclusion of $\VV$ in $\ov{\inv}$. Analogously, the composition $f\circ \phi_i: \UU\to S(d)_i$ coincides with the inclusion of $\UU$ in $S(d)_i$ and therefore  $\phi_i:\UU\to\VV$ is an isomorphism.
\end{dem}

As a consequence, we can state a version of [Corollary 6.4, \cite{CP}] for more general varieties and actions admitting positive-dimensional generic stabilizers.

\begin{corollary} Let $\g\subseteq L$ be a maximal semisimple subalgebra with $h^1(X,\TT\FF(\g))=0$. Then $\FF(\g)$ is rigid in \emph{$\overline{\inv}$}.
\end{corollary}
\begin{dem} Let $S(d)_i$ be the stratum passing through $\g$ in $S(d)$. Being $\g$ semisimple, its orbit is open. By Theorem  \ref{isolocal} it follows that $\FF(\g)$ is rigid in $\overline{\inv}$. 
\end{dem}

For arithmetically Cohen-Macaulay subvarieties $X\subseteq \PP^r$, we also state the above taking into consideration the (graded) homological dimension of $\TT\FF(\g)$.

\begin{corollary} \label{resFol} Let $X\subseteq \PP^r$ be smooth arithmetically Cohen-Macaulay of dimension $n$ and $\g\subseteq L$ a Lie subalgebra such that the morphism of sheaves $\g\otimes_\CC\OO_{X}\to\TT\FF(\g)$ admits a resolution of the form 
\[0 \rightarrow \bigoplus_{i=1}^{r_{n-2}}\OO_{X}(e^{n-2}_i)\rightarrow \cdots\rightarrow  \bigoplus_{i=1}^{r_{1}}\OO_{X}(e^{1}_i) \rightarrow \g\otimes_\CC\OO_{X}\rightarrow \TT\FF(\g)\rightarrow 0. \]
Then $\phi:\coprod_i S(d)_i\to \overline{\emph{\inv}}$ is an isomorphism locally around $\g$ and $\FF(\g)$. In particular, if $\g\subseteq L$ is rigid then $\FF(\g)$ is rigid in \emph{$\overline{\inv}$}. 
\end{corollary}
\begin{dem} It follows directly from Lemma \ref{resol} and Theorem \ref{isolocal}.
\end{dem}

\begin{ej} \label{invjrigido} \emph{The $j$-invariant in $\PP Sym^4(\CC^2)$ and the exceptional foliation}. The infinitesimal action of $\Sl_2$ on $\PP Sym^4(\CC^2)$ is locally free outisde a set of codimension 2. In particular, the morphism $\Sl_2\otimes_\CC \OO_{\PP^4}\to \TT\FF(\Sl_2)$ is an isomorphism. Since this is a semisimple Lie algebra, by the Theorem above the element $\FF(\Sl_2)$ is rigid in $\inv^1$. 

Since $\aff(\CC)\subseteq \Sl_4$ is also rigid, the same holds for the exceptional foliation $\FF(\g)$ in $\PP^3$ for $\g=\aff(\CC)$. This is, the foliation induced by the restriction of $\FF(\Sl_2)$ to the properly embedded $\PP^3\subseteq \PP Sym^4(\CC^2)$ is rigid in the corresponding moduli space $\inv^1$ of foliations in $\PP^3$.
\end{ej}

\begin{ej} \label{ejcuadrica} \emph{An action of $\aff(\CC)$ on the smooth quadric $Q\subseteq \PP^4$}. 
In \cite{fano3folds} the authors classify the irreducible components of $\FF^q(X,\omega_X)$ for Fano $3$-folds with $Pic(X)\simeq \ZZ$. In the case $X$ is the smooth quadric in $\PP^4$, this space has 3 irreducible components and only one of them consists generically of $L$-foliations, which are associated to a rigid action of $\aff(\CC)$. We will now see that we can recover its rigidity in $\inv$ as an application of the above results.

Let us identify the space $\PP Sym^4(\CC^2)$ with the set of effective divisors in $\PP^1$. The closure of the point $q=1+(-1)+i+(-i) $ under the action of $\Sl_2$ in the example above is a smooth quadric $Q$. Let us consider the foliation induced by the restriction of this action to the affine Lie algebra $\aff(\CC)$ fixing the point $\infty$. Since its orbits are of dimension $2$ outside the set $\{ p+3\infty\} \cup \{ 3p+\infty\}$ we have
\[ \TT\FF(\aff(\CC))\simeq \aff(\CC)\otimes_\CC\OO_Q .\]
In particular, we are under the hypotheses of Corollary \ref{resFol}. Recall that in this case we have $L=\mathfrak{so}_5$. Using the GAP software we can easily compute the dimensions corresponding to the inclusion $\aff(\CC)\subseteq \Sl_2\subseteq\so_5$:
\[ \dim Z^1(\aff(\CC),\so_5/\aff(\CC))=\dim B^1 (\aff(\CC),\so_5/\aff(\CC))=8. \]
It follows that this is a rigid subalgebra and therefore the foliation $\FF(\aff(\CC))$ is rigid in the space $\inv^1$ of foliations on $Q$.
\end{ej}

\begin{ej} \label{adjsln} \emph{The adjoint representation of $\Sl_n$. }
Let $\g=\Sl_n\curvearrowright V$ be the adjoint representation. Let us consider the foliation $\FF(\g)$ associated to the corresponding action on $\PP V$. Its orbits coincide with the conjugacy classes of $sl_n$ under the action of $SL_n$. Its algebra of invariants is the polynomial algebra on the $n-2$ coefficients of the characteristic polynomial $\chi$. This is, the generic leaf of this foliation coincides with the generic fiber of the map
$$ \chi: \PP V \dashrightarrow \PP^{n-3}_{(2,\dots,n)}.$$
The inclusion $\Sl_n\hookrightarrow H^0(\PP V, \TT\PP V)$ can be interpreted as follows: for every element $\X\in \Sl_n$, the evaluation at a point $[p]\in \PP V$ is   
\[ \X([p])=\ov{[\X,p]}\in V/\langle p \rangle\simeq\TT_{[p]}\PP V.\]
With this in mind, we can compute the kernel of the morphism 
$$\rho: \Sl_n\otimes \OO_{\PP V}\to \TT \FF(\g).$$
Let us consider the Zariski open sets of $\PP V$ 
\[ \UU=\{ [p]\in \PP V \hspace{0.2cm}\vert\hspace{0.2cm} \{ Id, p,\dots, p^{n-1}\} \mbox{ is linearly independent}\}. \]
and $ \WW=\{ [p]\in \PP V \hspace{0.2cm}\vert\hspace{0.2cm} p^n\neq0 \}$. It is quite clear that $\codim(\PP V\setminus \WW)\geq 2$. Let us see that the same holds for $\UU$. First, observe that the points lying in $\UU$ are the classes of matrices whose minimal and characteristic polynomials coincide, or equivalently whose Jordan canonical forms have a unique Jordan block for each of its distinct eigenvalues. In particular, $\PP V\setminus \UU$ is contained in the discriminant divisor $D$ consisting of (classes of) elements with less than $n$ distinct eigenvalues. We claim that $\UU\cap D$ is dense in $D$. Indeed, for every $[p]\in D$ there exists an element $v\in V$ such that $[p+t v]\in D\cap \UU$ for every  $t\in \CC^*$: if $p$ has Jordan canonical form 
$$p=c^{-1} (d+u) c$$ for an invertible element $c\in V$ and $d$ and $u$ diagonal and strictly upper-triangular matrices respectively, then straightforward calculation shows that we can take $v=c^{-1}(s-u) c$ where $s=(s_{ij})$ satisfies $s_{ij}=1$ if $j=i+1$ and $s_{ij}=0$ otherwise. This shows that $\PP V\setminus \UU$ has positive codimension in $D$ and therefore $\codim(\PP V\setminus \UU)\geq 2$.

We will now compute the kernel of the restriction $\rho\vert_{\UU\cap \WW}$.
For a point $[p]\in \WW$ and an element $\X\in \Sl_n$, we have $\X([p])=0$ if and only if $[\X,p]=0$. Indeed, if this were not the case then there would exist an element $A\in \Sl_n$ such that  
\[ [A,p]=p.\]
By induction one can see that this easily implies that $[A,p^k]=n_k p^k$ for some $\{ n_k\}_{k\in\NN}\subseteq \NN$. But then $tr(p^k)=0$ for every $k\in \NN$ and therefore $p$ must be a nilpotent element, which leads to a contradiction. 
This means that the fiber $\ker(\rho)([p])$ over some point $[p]\in \WW$ is the centralizer of $p$.

For every $1\leq k \leq n-1$ consider the matrix $P_k\in \CC[V]_k^{n\times n}$ such that for every $p\in V$ we have
 \[ P_k(p)=p^k-tr(p^k)Id .\]
Each of these elements defines a section $Z_k\in H^0( \PP V, \Sl_n\otimes \OO_{\PP V}(k))$ that trivially satisfies $Z_k\in\ker(\rho)$. Observe that these sections are linearly independent around every $[p]\in \UU$. Moreover, it follows from [Theorem 3.2.4.2 \cite{HJ}] that these also generate $\ker(\rho)$ around every $[p]\in \UU\cap \WW$.

We have now constructed a complex 
\[ 0\rightarrow \bigoplus_{k=1}^{n-1} \OO_{\PP V}(-k) \to \Sl_n\otimes \OO_{\PP V}\rightarrow \TT \FF(\g) \rightarrow 0 \]
that is exact on $\UU\cap \WW$. Since $\codim(\PP V\setminus (\UU\cap \WW))\geq 2$, this means that this is in fact a resolution of $\TT \FF(\g)$. 
Being $\Sl_n$ semisimple, if $n\geq 3$ we can apply Corollary \ref{resFol} in order to conclude that the foliation $\FF(\g)$ is rigid in $\inv$.

\end{ej}

\begin{ej} \label{codigoM2}Let us consider the foliation of codimension $2$ on $\PP^5$ induced by the vector fields
\begin{align*} \X_1&= (x_1+x_2)\frac{\partial}{\partial x_1} + x_2 \frac{\partial}{\partial x_2}+ (x_4+x_5)\frac{\partial}{\partial x_4} + x_5 \frac{\partial}{\partial x_5}, \\
\X_2 &=-x_0\frac{\partial}{\partial x_1} -x_3\frac{\partial}{\partial x_4}\mbox{\hspace{0.2cm} and} \\
\X_3 &= - x_0 \frac{\partial}{\partial x_1}-x_0 \frac{\partial}{\partial x_2} -  x_3 \frac{\partial}{\partial x_4}-x_3 \frac{\partial}{\partial x_5}
\end{align*}
(this is, the foliation induced by the subalgebra $\g\subseteq L$ of Example \ref{calcfamilia1} and Example \ref{familia1c}). Since  $\g\otimes_{\CC}\OO_{\PP^5}\to\TT\FF(\g)$ is an isomorphism, this subalgebra satisfies the hypotheses of Corollary \ref{resFol}. Recall from Example \ref{calcfamilia1} that the space tangent to $S(3)$ at the point $\g$ is of dimension $32$ and that the dimension of the tangent space of its orbit under the action of $\Aut(\PP^5)$ is $28$. Applying Corollary \ref{resFol}, we can conclude that every irreducible component of $\inv$ passing through $\g$ is of dimension less or equal than $32$ (and greater or equal than $28$). 
\end{ej}

In the case of foliations of codimension $1$ we can translate the above results to the spaces $\FF^1(X,\LL)$, provided that $Pic(X)$ is discrete: 

\begin{teo} \label{teoformas} Suppose $Pic(X)$ is discrete and let $\g\subseteq L$ be a maximal subalgebra acting with orbits of dimension $\dim(X)-1$ in codimension $1$ and such that $h^1(X,\TT\FF(\g))=0$. If $Sing(\FF(\g))$ is without embedded components and   $Sing(\FF(\g))=\ov{K(\FF(\g))}\cup\{p_1,\dots,p_r\}$ where the $p_i$'s are Reeb-type singularities, then there exist neighborhoods $\UU\subseteq \coprod_i S(d)_i$ and \emph{$\VV\subseteq \FF^1(X,N_{\FF(\g)}^{\vee\vee})$} of $\g$ and $\FF(\g)$ respectively such that the morphism $\omega:\UU\to\VV$ defined by $\g\mapsto [\omega(\g)]$ is an isomorphism. In particular, if $\g$ is a rigid subalgebra then $\FF(\g)$ is rigid in $\FF^1(X,N_{\FF(\g)}^{\vee\vee})$.
\end{teo}
\begin{dem} Let $\g\subseteq L$ be a subalgebra satisfying the hypotheses of the Theorem and $\LL=\det(N_{\FF(\g)})=N_{\FF(\g)}^{\vee\vee}$. Let $\FF^1_\LL$ be the tautological family of foliations with base $\FF^1(X,\LL)$ constructed in Remark \ref{familiataut}. The hypothesis on $Sing(\FF(\g))$ allow us to apply [Theorem 8.13, \cite{Quall}] in order to ensure that the singular scheme $Sing(\FF^1_\LL)$ of the family is flat in a neighborhood of $[\omega(\g)]\in \FF^1(X,\LL)$. Then, by [Theorem 7.8, \cite{Quall}] the normal sheaf  $N_{\FF^1_\LL}$ is also flat in a neighborhood of $[\omega(\g)]$. This is, the stratum of the correponding flattening stratification passing through $[\omega(\g)]$ is open in $\FF^1(X,\LL)$. Using Theorem \ref{isolocal} and Proposition \ref{compararinv} we can conclude that the composition 
\[ \omega: \coprod_i S(d)_i \rightarrow \inv^1\rightarrow \FF^1(X,\LL)\] 
is an isomorphism locally around $\g$ and $\FF(\g)$.
\end{dem}
\begin{corollary}\label{corformas} Let $X$ be arithmetically Cohen-Macaulay of dimension $n$ and $\g\subseteq L$ a subalgebra acting with orbits of dimension $n-1$ in codimension $1$. Suppose further that the morphism $\g\otimes_\CC \OO_{X}\to \TT\FF(\g)$ admits a resolution of length $n-2$ of the form
\[0 \rightarrow \bigoplus_{i=1}^{r_{n-2}}\OO_{X}(e^{n-2}_i)\rightarrow \cdots\rightarrow  \bigoplus_{i=1}^{r_{1}}\OO_{X}(e^{1}_i) \rightarrow \g\otimes_\CC\OO_{X}\rightarrow \TT\FF(\g)\rightarrow 0. \]
If $Sing(\FF(\g))$ is without embedded components and  $Sing(\FF(\g))=\ov{K(\FF(\g))}\cup\{p_1,\dots,p_r\}$ where the $p_i$'s are Reeb singularities, then there exist neighborhoods $\UU\subseteq \coprod_i S(d)_i$ and \emph{$\VV\subseteq \FF^1(X,N_{\FF(\g)}^{\vee\vee})$} of $\g$ and $\FF(\g)$ respectively such that the morphism $\omega:\UU\to\VV$ defined by $\g\mapsto \omega(\g)$ is an isomorphism. In particular, if $\g$ is a rigid subalgebra then $\FF(\g)$ is rigid in $\FF^1(X,N_{\FF(\g)}^{\vee\vee})$.
\end{corollary}
\begin{dem} It follows from the Theorem above and Lemma \ref{resol}.
\end{dem}

As was already mentioned, the deformation theory of $\g$ and $\FF(\g)$ are governed by the Chevalley-Eilenberg complex $\C^\bullet(\g,L/\g)$ and the leaf complex $\LL_{\FF(\g)}$ respectively.
We will end this section by giving a first order intepretation of Theorem \ref{isolocal} in terms of these complexes. We will denote the scheme of dual numbers by $D=Spec(\CC[\varepsilon]/(\varepsilon^2))$.

Let $\g\subseteq L$ be a maximal subalgebra such that $h^1(X,\TT\FF(\g))=0$ and $S(d)_i$ the stratum passing through $\g$. The derivative of $\phi$ at a point $\g$ is the map
\begin{align*}d_\g\phi: \TT_\g S(d)_i\subseteq \SUB^L_d(Spec(D)) &\to \INV(Spec(D))=\TT_{\FF(\g)} \inv\\
(\GG\subseteq L_D)&\mapsto \FF(\GG),
\end{align*}
where $\TT_\g S(d)_i$ is the set of families of subalgebras $\GG\subseteq L_D$ over $D$ such that $\GG(0)=\g$ and whose  normal sheaf $N_{\FF(\GG)}$ is flat over $D$. By Remark \ref{secciones} every such family satisfies $\pi_{2,*}\TT\FF(\GG)=\GG$. In particular this implies that $d_\g\phi$ is injective. Just as in the proof of Theorem \ref{isolocal}, we can use our hypotheses in order to construct a section for this map (but this time in terms of the corresponding complexes). By Proposition \ref{tangentelie} and Theorem \ref{tangente} we can realize $d_\g \phi$ as a map
\[ V_i\subseteq Z^1(\g,L/\g)\to \ker(H^0(d_1)) \]
where $V$ is the subspace corresponding to $\TT_\g S(d)_i\subseteq \TT_\g S(d)$ and $d_1$ is the differential of $\LL_\FF$. 
Observe that the hypotesis on the first cohomology group of $\TT\FF(\g)$ implies that the cohomology of $\FF(\g)$ is of the form
\[ 0\rightarrow \g \rightarrow L \rightarrow H^0(X,N_{\FF(\g)})\rightarrow 0\rightarrow \cdots \]
and therefore $H^0(X,N_{\FF(\g)})\simeq_{\g-mod} L/\g$. 

On the other hand, the exterior powers of the morphism $\g\otimes\OO_X\to \TT\FF(\g)$ are epimorphisms 
\[ f^\bullet: \bigwedge^\bullet \g \otimes_\CC \OO_{X} \rightarrow \bigwedge^\bullet \TT\FF(\g) .\]
In order to simplify the notation we will use the symbol $f$ for any of these maps. Their duals 
\[ f^*: \Hom(\bigwedge^\bullet \TT\FF(\g),N_{\FF(\g)})\to \Hom(\bigwedge^\bullet \g\otimes_\CC \OO_{X}, N_{\FF(\g)})\subseteq \Hom_\CC(\bigwedge^\bullet \g,L/\g)\]
induce a monomorphism of complexes $f^*: H^0(L_\FF)\to C(\g,L/\g)$
of the form
\[\begin{tikzcd}
 &0 \arrow[d] & 0 \arrow[d] & \\
 L\arrow[r]\arrow[d,"id" ] \arrow[d, equals] & \Hom( \TT\FF(\g),N_{\FF(\g)}) \arrow[r, "H^0(d)"]\arrow[d, "f^*"]& \Hom(\bigwedge^2 \TT\FF(\g),N_{\FF(\g)})\arrow[d, "f^*"]\arrow[r, "H^0(d)"] & \cdots \\
L \arrow[r] & \Hom_\CC( \g, L/\g)\arrow[r, "\delta"]& \Hom_\CC(\bigwedge^2 \g, L/\g) \arrow[r, "\delta"] & \cdots
 \end{tikzcd}\]
Being $f^*$ a morphism of complexes we must have $f^*\ker(H^0(d))\subseteq \ker (\delta)$. This is, $f^*$ is a transformation
\[f^*: \ker(H^0(d_1))\to Z^1(\g,L/\g). \]
Now let $\FF_D$ be a first order deformation of $\FF(\g)$ and $\eta\in \ker(H^0(d_1))$ the corresponding element according to Theorem \ref{tangente}. Its pullback $f^*\eta:\g\to L/\g$ defines the family of subalgebras of $L$ over $D$ given by 
\[ \GG=H^0(X_D,\TT\FF_D)\hookrightarrow L_D.\]
The differential of $\phi$ applied to this element is the family with tangent sheaf 
$$\im(\rho_\GG:\GG \otimes \OO_{X\times D}\to \TT X_D)\subseteq \TT\FF_D.$$ 
Being $\TT\FF(\g)$ globally generated, this inclusion is in fact an equality. We can then conclude that $f^*\eta$ lies in $V_i$ (this is, it is tangent to the stratum $S(d)_i$) and that $f^*$ is in fact a section of the differential of $\phi$ at the point $\g$, which is therefore bijective.

\begin{obs} The above discussion would actually prove that $\phi$ is an isomorphism locally around $\g$ and $\FF(\g)$, provided that $\g$ or $\FF(\g)$ are smooth points of the schemes $\coprod S(d)_i$ and $\overline{\inv}$.
\end{obs}

\end{document}